\documentclass[11pt]{amsart}
\usepackage{latexsym,amsmath,amssymb,epsfig,graphicx}
\usepackage{caption}
\usepackage{cellspace}
\usepackage[utf8]{inputenc}
\usepackage[greek.polutoniko,english]{babel}
\usepackage{wasysym}
\usepackage{wrapfig}
\usepackage{float}
\newtheorem{thm}{Theorem}

\newtheorem{conj}{Conjecture}

\def\C{\mathbb C}

\def\P{\mathbb P}

\usepackage[left=1.4in,right=1.4in,top=.95in]{geometry}

\begin{document}
\title[KP1 ``lump'' solutions from cuspidal curves]{An algebraic-geometric construction of ``lump'' solutions of the KP1 equation}
\author{John B. Little}
\address{Professor {\it emeritus}, Department of Mathematics and Computer Science,
College of the Holy Cross, Worcester, MA 01610}
\email{jlittle@holycross.edu}

\date{\today}

\maketitle
\begin{abstract}
In this note, we show how certain everywhere-regular real rational function solutions of the KP1 equation (``multi-lumps'') can be 
constructed via the polynomial analogs of theta-functions from singular rational curves with cusps.  The method we use 
can be understood either directly, or as producing a degeneration of the well-understood soliton solutions from nodal singular curves.  From
the second of these points of view, it can be
seen as a variation on the \emph{long-wave limit} technique of Ablowitz and Satsuma from \cite{AS,SA}, as developed
by Zhang, Yang, Li, Guo, and Stepanyants, \cite{Zetal}.  Our main explicit example is a three-lump solution 
constructed via the polynomial analog of the theta-function from a rational curve with two cuspidal singular points, each
 with semigroup $\langle 2,5\rangle$.  (In the theory of curve singularities, these are known as $A_4$ double points.)
We show that there are similar six-lump solutions from a rational curve with two cusps, each with semigroup $\langle 2,7\rangle$ ($A_6$ double points).  
We conjecture that these ideas will generalize to give similar $M$-lump solutions with $M = \frac{N(N+1)}{2}$ for $N > 2$ starting
from rational curves with two singular points with semigroup $\langle 2,2N+1\rangle$ ($A_{2N}$ double points).   Finally 
we present a five-lump solution constructed from a rational curve with two cusps, each with semigroup $\langle 3,4\rangle$.
Similar solutions have been constructed by other methods previously; our contribution is to show how
they arise from the algebraic-geometric setting by considering singular curves with several cusps, as in \cite{ACL}.  
\end{abstract}

\section{Introduction}
We will consider the Kadomtsev-Petviashvili (KP) equation for $u = u(x,y,t)$ in the form
\begin{equation}\label{KP12}
(-4 u_t + 6uu_x + u_{xxx})_x \pm 3u_{yy} = 0.
\end{equation}
Solutions of these PDEs are of considerable interest in physics since they model several different sorts of 
wave phenomena in two space dimensions and time. 
In the PDE literature, taking the $+$ sign on the final term gives what is called the KP2 equation, while the $-$ sign gives the KP1 
equation.  The differences between these cases are often not emphasized by authors discussing the construction
of solutions via the algebraic-geometric techniques studied here.  This is no doubt true because solutions of one form of the
 equation can be taken to solutions of the 
other by a complex rescaling of the independent variables taking $y \mapsto i\cdot y$.  However, the behavior of the solutions for real values of the 
space variables in the two cases is quite distinct.  In particular, the regular real rational function solutions that we are interested in
arise only in the KP1 case.   

The work of Mikio Sato and his school on construction
of solutions via $\tau$-functions corresponding to points of the infinite-dimensional Sato Grassmannian, \cite{S}, has 
illuminated the structure of solutions and shown close connections with combinatorial constructions such as 
Young diagrams for partitions and Schur polynomials.  

Various classes of solutions to KP1 and KP2 (and other related soliton equations) using techniques from algebraic geometry 
have been known since the late 1970's, based on the so-called Krichever construction, which shows how to 
produce points of the Sato Grassmannian starting from data from an algebraic curve.  
As a result, the KP equation has  generated a great deal of interest
in algebraic geometry.  In the most spectacular connection, KP solutions play a key role in the results of Shiota, 
Mulase,  and many others on
the Schottky problem of characterizing the Jacobian varieties of smooth curves among all principally polarized
abelian varieties.  In the algebraic-geometric context, the solutions produced from smooth curves via theta-functions
on their Jacobians have received the greatest amount of attention.  In the PDE context, these are the so-called 
quasiperiodic solutions.   

However, it has long been understood--and the details have become increasingly clear--that other classes of KP solutions have connections with 
other classes of singular curves in a very parallel manner.  For instance, one of the classes of solutions that 
will enter for us at one point are the \emph{soliton} solutions.  These arise in the algebraic-geometric
context by considering the limit of the theta-function as a family of smooth curves of genus $g$ degenerates to 
an irreducible rational curve with $g$ ordinary double points (``nodes'').  The connections here were glimpsed very early
and exploited by Mumford for the construction of KdV and KP solitons  in \cite{M}.   They have been studied in much more detail recently in the works of many 
authors.  The most relevant for our purposes are the papers by Agostini, Fevola, Mandelshtam, and Sturmfels, \cite{AFMS}, 
and the related work of Fevola and Mandelshtam, \cite{FM}.  The fact that \emph{all} real regular KP soliton solutions corresponding
to  $\tau$-functions from points of the totally nonnegative Grassmannian can be expressed by the theta-functions on 
nodal singular curves has recently been established by Kodama in \cite{K}.   

The connection between rational function KP solutions, Schur polynomials and 
polynomial analogs of theta-functions from cuspidal singular curves also has a long history.  We mention
in particular the examples from Pelinovskii and Stepanyants, \cite{PS}, and the foundational article of Buchstaber,  Leykin, and Enolski, \cite{BLE}.  More
recently, this connection been studied in the article \cite{ACL}, which gives more details about the relation between
the polynomial analogs of theta-functions for cuspidal curves and rational KP solutions.  We
will make use of several crucial results from that article.   We will also follow many of the notational and
terminological conventions for singular curves established there, so readers may wish to consult
\cite{ACL} for background material.  

Unfortunately, from the point of view of applied PDE, ``most of'' the rational KP solutions produced from cuspidal
singular curves as in \cite{ACL} are probably of relatively little interest because they tend to be non-regular 
at some real $(x,y)$ for some or all real $t$.   This is because the denominator of the rational solution 
\begin{equation}\label{KPu}
u(x,y,t) = 2 \frac{\partial^2}{\partial x^2} \ln(\tau(x,y,t))
\end{equation}
will ``usually'' vanish for some real $(x,y,t)$ when $\tau(x,y,t)$ is a polynomial with real coefficients. The same will 
be true more generally if $\tau$ is the product of an exponential factor with
exponent linear in $x$ and a polynomial as in Theorem 4.11 from \cite{ACL}.  But in fact we will be somewhat sloppy about the terminology here
and essentially ignore any such exponential factor that might be present, calling the polynomial factor itself the $\tau$-function.  
The reason this is harmless from our point of view  is that in passing from $\tau(x,y,t)$ to $u(x,y,t)$ via \eqref{KPu} the exponential factor contributes nothing to the 
actual KP solution.  A simple observation here is that (the polynomial part of) $\tau(x,y,t)$ must have even total degree in $(x,y)$
in order to obtain everywhere-regular real KP solutions.  

The complex rescaling $y \mapsto i\cdot y$ to go from a KP2 solution to a KP1 solution
can also produce $\tau(x,y,t)$ and $u(x,y,t)$ that take non-real values for
some real $(x,y,t)$.  These solutions are also of less interest for applications.

From this point of view, possibly the most interesting rational KP solutions are 
rational ``lump'' solutions--rational functions $u(x,y,t)$ which are real-valued for all real $x,y,t$, whose denominators never vanish for real $(x,y,t)$,
and which decay to $0$ in all directions for all $t$.  Motivated by questions concerning ``rogue waves''
and other actual physical phenomena, quite a few such solutions have been constructed to date by several authors, and in several
different ways.  

In this note, we will produce several examples of  regular multi-lump rational solutions of KP1 by 
using the constructions from \cite{ACL}.    Our first example will be a three-lump solution.  A similar
computation yields a six-lump solution.  In both of these, outside a central range of $t$-values, the lumps appear in triangular
configurations and undergo an interesting interaction in that central $t$-range.  We conjecture that the constructions used for these will generalize to give 
triangular $N(N+1)/2$-lump solutions for all $N\ge 1$.   Finally, we will present a somewhat
similar five-lump solution obtained by the same methods.  

From the results of \cite{ACL}, we have
a precise recipe for producing a point of the Sato Grassmannian
corresponding to a cuspidal curve.  We also know  how the 
corresponding $\tau$-function for the KP solution is related to the polynomial
analog of the theta-function for the cuspidal curve.  To be clear, we note that the results of \cite{ACL} are
geared toward producing solutions of the KP2 equation.  Hence we must apply the 
complex rescaling $y \mapsto i\cdot y$ to get a KP1 solution.  We then conclude
by finding a $\tau$-function, hence a solution $u(x,y,t)$,  that is real for all real $(x,y,t)$.
We also show that our solution is regular for all real $(x,y,t)$ by expressing the $\tau$-function
as a sum of squares of real polynomials with a nonzero constant term. 

We will see that the $\tau$-function be computed directly, or by using the idea of degenerating a nodal curve to 
a cuspidal curve  (hence degenerating a soliton solution to a rational function solution) in a 
particular well-chosen way.  In this connection, we begin by mentioning the early work of Ablowitz and Satsuma, \cite{AS,SA},
which showed how to produce regular rational solutions from solitons by what they called
the ``long-wave limit'' process.  In our terms, it can be seen easily that their method amounts to 
taking the theta-function from an irreducible rational nodal curve and determining what happens when the
nodes degenerate to \emph{ordinary cusps} (i.e. singular points analytically isomorphic to the origin on 
$y^2 - x^3 = 0$, or so-called $A_2$ double points).  In our terms, the nodes come by identifying pairs of points $\{b,c\}$ on the 
normalization.  We always assume our curves are irreducible and rational, so this is just  the complex 
projective line $\mathbb{P}^1$ (i.e. the Riemann sphere).  Then Ablowitz and 
Satsuma's construction amounts to taking a limit as the pairs $\{b,c\}$ coalesce to single points.   Special
choices must then be made to ensure that the limit of the soliton solution is regular, essentially
by making the limiting polynomial $\tau$-function expressible as a \emph{sum of  squares} of 
real polynomials with a positive nonzero constant term.
Those conditions have also been realized by using other methods to find the limiting solutions,
most notably the Gram matrix techniques used by Chakravarty and Zowada in \cite{CZ1,CZ2}, and 
the perturbation processes described by Zhang, et al. in \cite{Zetal}.

In our three-lump example, the soliton solution
comes from the theta-function of an irreducible nodal curve of arithmetic genus $g=4$, obtained by identifying
\emph{four} pairs of points $\{b_i,c_i\}$, $i = 1,\cdots,4$,  on $\mathbb{P}^1$.  By letting
$\{b_1,c_1,b_2,c_2\}$ coalesce in a particular way, we produce a singular point
with semigroup $\langle 2,5\rangle$ (an $A_4$ double point).  Simultaneously, 
$\{b_3,c_3,b_4,c_4\}$ coalesce to a second, distinct but analytically equivalent singular point,
also with semigroup $\langle 2,5\rangle$.  From \cite{ACL}, we know that the total
degree of the polynomial theta-function will be $6 = 3 + 3$ in this case because
the Young diagram for each of the singular points is the triangular diagram
corresponding to the partition $(2,1)$.   The connection with these particular Young diagrams and 
the necessity of degenerating to a curve with \emph{two} cuspidal singular points
was suggested by the results of \cite{CZ1,CZ2}.  

 Throughout 
this work, we will report results whose derivations made rather heavy use of symbolic 
computation.  We used  the {\tt Maple 2024} computer algebra system, \cite{Ma}, to work with some of the
rather complicated formulas involved.  

\vskip 10pt
{\sc Acknowledgments.}  This work is essentially a continuation of \cite{ACL} and I would like to thank 
Daniele Agostini and T\"urk\"u \"Ozl\"um \c Celik for a number of helpful conversations as this was developing.  
I would also like to thank Bernd Sturmfels once again for facilitating  contact with Daniele and T\"urk\"u  starting in 
2020 and for renewing my interest in algebraic curves and applications  to PDE.

\section{Background on cuspidal singular curves and the polynomial analogs of theta-functions}

We begin by recalling several results from \cite{ACL} that will be the basis for a direct method for 
constructing lump solutions.   See the references there for general background on singular curves.  

Let $C$ be a reduced, irreducible rational curve of arithmetic genus $n$, 
Let $\omega_C$ be the sheaf of dualizing differentials (an invertible sheaf if $C$ is Gorenstein).
The global sections of $\omega_C$ can be described as follows.  Consider the 
normalization $\nu : \P^1 \to C$.  The elements of $H^0(C,\omega_C)$ are 
 meromorphic differentials $\omega = g(v)\ dv$ on $\P^1$ with 
\begin{equation}
\label{RosenlichtDiff}
\sum_{q \in \nu^{-1}(p)} {\rm Res}_q(\nu^*(f)\cdot \omega) = 0
\end{equation}
for all $p \in C$ and all $f \in {\mathcal O}_p$.
If $p$ is unibranch, and $q = \nu^{-1}(p)$ then we must have ${\rm Res}_q(\omega) = 0$.  

The generalized Jacobian of $C$ is defined as the quotient $J(C) = H^0(C,\omega_C)^*/H_1(C,\mathbb{Z})$.
Let $C_0 \subset C$ be the  smooth locus.  Fix $P_0\in C_0$.
The \emph{Abel mapping} with basepoint $P_0$ is   $\alpha : C_0 \to J(C)$ defined by $P \mapsto (\omega \mapsto \int_{P_0}^P \omega)$.
This extends by linearity to a mapping on effective divisors of all degrees $d$:  $\alpha^{(d)} : C_0^{(d)} \to J(C)$.
If (arithmetic) genus $g(C) = n$, the subvariety $W_{n-1}$, that is, the closure of the 
image of $\alpha^{(n-1)}$, will be called the \emph{theta-divisor} of $C$ and we will consider an implicit equation
for the theta-divisor as an analog of a \emph{theta-function} for $C$.
The generalized Jacobian $J(C) \cong \C^n$ and the theta-divisor (considered as a hypersurface in $\C^n$) is defined
by a polynomial equation if and only if $C$ is rational and all singular points of $C$ are ``cuspidal'' (that is, analytically irreducible, or unibranch).

In \cite{ACL}, a number of general facts concerning the form of the polynomial analogs of theta-functions were derived for cuspidal curves.  
These depend on two different invariants of the cuspidal singularities.  The first of these is the \emph{semigroup} of the singular
point $P$, the set of non-negative integers that are vanishing orders of  elements of the local ring ${\mathcal O}_P$ at $P$.  This a 
subsemigroup of $\mathbb{N}_{\ge 0}$ under addition, and its complement in $\mathbb{N}_{\ge 0}$ is finite.  

\begin{thm} [\cite{ACL}, Theorem 1.1]
\label{cuspidal}
If $C$ is Gorenstein, rational, of genus $n$ with only cuspidal singularities,
the degree of the polynomial analog of the theta-function giving an implicit equation of $W_{n-1}$ is $\le n(n+1)/2$.  
The bound is attained for curves with a single singular
point with semigroup $\langle 2, 2n + 1\rangle$.
\end{thm}

The complement of the semigroup is
known as the set of \emph{gaps} and the number of gaps is denoted by $\delta_P$.  When a curve has several singular points, 
its arithmetic genus is the sum $\sum_{P \in C \setminus C_0} \delta_P$.  

Since the set of gaps is finite, there will be a
minimal integer $d_P$ such that all $n \ge d_P$ are contained in the semigroup.  The Gorenstein condition is equivalent 
to the equation $d_P = 2\delta_P$ at all singular points.   We will only consider singularities satisfying this condition in this note.

If the gaps at $P$ are $w_1 < w_2 < \cdots < w_n$, then the second invariant of the singularity that we will use 
is the \emph{Weierstrass partition} $\lambda_P = (\lambda_1, \ldots, \lambda_n)$ associated to $P$.  The parts of $\lambda_P$ are defined by
\begin{equation}
\label{WP}
\lambda_i = w_{n+1 - i} - (n - i), \text{ for } 1 \le i \le n.
\end{equation}

\noindent
{\bf Example.}  
Consider the rational monomial curve $C$ parametrized by
\begin{align*}
\nu : \mathbb{P}^1 &\longrightarrow \mathbb{P}^3\\
(t:s) &\longmapsto (x_0 : x_1 : x_2 : x_3) = (t^6: t^2 s^4: t s^5 : s^6).
\end{align*}
$C$ is a curve of degree $6$ and arithmetic genus $4$ with implicit equations 
$$x_1 x_3 - x_2^2 = x_1^3 - x_0 x_3^2 = 0.$$
It has exactly one singular point $P= \nu(1:0) = (1:0:0:0)$.  
Using the affine coordinate $v = t/s$ on $\P^1$,  the singularity is at $v = \infty$ and $P_0$, defined by $v = 0$, is a smooth
point.   The local ring ${\mathcal O}_{P}$ generated by $(s/r)^{4}, (s/t)^{5}, (s/t)^{6}$, or $v^{-4},v^{-5},v^{-6}$. 
That singular point is clearly \emph{unibranch} with semigroup $= \langle 4,5,6\rangle$.  The gaps 
are $\{1,2,3,7\}$.  Note that $\delta = 4, d = 8$, so this is Gorenstein.
By \eqref{WP}, the partition corresponding to the singularity at $P$ is 
$\lambda = (4,1,1,1)$.
By \eqref{RosenlichtDiff}, a basis for $H^0(C,\omega_C)$:
is $\{ dv,\ v\ dv,\ v^2\ dv,\ v^6\ dv \}$.  The theta-divisor is parametrized by
\begin{align*}
Z_1 &= \int_0^{t_1} dv + \int_0^{t_2} dv + \int_0^{t_3} dv \\
Z_2 &= \int_0^{t_1} v\ dv + \int_0^{t_2} v\ dv  + \int_0^{t_3} v\ dv\\
Z_3 &= \int_0^{t_1} v^2\ dv + \int_0^{t_2} v^2\ dv  + \int_0^{t_3} v^2\ dv\\
Z_4 &= \int_0^{t_1} v^6\ dv  + \int_0^{t_2} v^6\ dv + \int_0^{t_3} v^6\ dv.
\end{align*}
These abelian integrals are also \emph{polynomials}, so the 3-fold parametrized above has a polynomial equation.
We can find it, for instance, using Gr\"obner bases.  Carrying out the integration, then eliminating $t_1,t_2,t_3$, 
we get the implicit equation
\begin{equation}
\label{4111theta}
Z_1^7+21Z_1^4Z_3-84Z_1^3Z_2^2+252Z_1 Z_3^2+252 Z_2^2 Z_3-252Z_4 = 0.
\end{equation}
This is the polynomial analog of the theta-function in this case
and the pattern described in Theorem~\ref{cuspidal}  above is clear since $7 < (4\cdot 5)/2 = 10$.  $\triangle$

\vskip 10pt
In general, the total degree of the polynomial analog of the theta-function for a Gorenstein cuspidal curve
and the form of its highest degree term are described as follows.

\begin{thm}[\cite{ACL}, Theorem 5.3]
\label{moredetail}
Let $C$ be a curve with several Gorenstein cuspidal singularities $P_1,\ldots, P_h$ with corresponding 
partitions $\lambda(P_1), \ldots, \lambda(P_h)$.  Take a basis of $H^0(C, \omega_C)$
 that is a union of the local bases at the $P_i$.  For each $i$, let  $Z_1^{(i)}$ correspond to the basis
 differential with minimal pole order at $P_i$.   Then the highest-degree monomial
in the polynomial analog of the theta-function is
$$(Z_1^{(1)})^{|\lambda(P_1)|} \cdots (Z_1^{(h)})^{|\lambda(P_h)|},$$
total degree is
$\deg(\theta) = \sum_{i=1}^h |\lambda(P_i)|$.  
\end{thm}

In the foregoing example, note $\lambda = (4,1,1,1)$.  Hence $| \lambda | = 7$ and the highest degree term
in \eqref{4111theta} is $Z_1^7$.  The polynomial there is also homogeneous for the weighting 
$${\rm wt}(Z_1) = 1, {\rm wt}(Z_2) = 2, {\rm wt}(Z_3) = 3, {\rm wt}(Z_4) = 7.$$
These weights are the gaps of the semigroup $\langle 4,5,6\rangle$.  This homogeneity is a 
special feature coming from the $\mathbb{C}^\ast$ torus action on the monomial curve.

\section{From the theta-function to a KP tau-function}

There is a close connection between the polynomial analogs of theta-functions for cuspidal curves as developed in
the previous section and tau-functions for the KP hierarchy.  The underlying reason for this is that the 
Krichever construction producing a point of the Sato Grassmannian from data connected to an algebraic
curve also applies in this situation of cuspidal singular curves.  See \cite{ACL}, Proposition 3.1 for the exact 
statement.   In particular, there is 
a frame for the subspace corresponding to the Grassmannian point where the columns at the right of the frame matrix
come from the series expansions of the differentials in a basis of $H^0(C,\omega_C)$.  This yields the 
following statement.

\begin{thm}[(ACL), Theorem 3.5, Theorem 4.11]
\label{thetatotau}
Let $C$ be a Gorenstein cuspidal curve of arithmetic genus $g$.  Let $\omega_j = \sum_{i=0}^\infty a_{ij} u^i$ be
the local expansions of the differentials in a basis of $H^0(C, \omega_C)$ at the smooth point $P_0$, 
and let $A$ be the $g \times \infty$ transposed matrix $A = (a_{ji})$.  Let ${\bf t} = (t_1,t_2, \ldots\ )$
be an infinite collection of indeterminates.  Then a KP tau-function corresponding to the 
point of the Sato Grassmannian obtained from $C$ by the Krichever construction is given by 
$$\tau({\bf t}) = e^{\ell({\bf t})} \theta(A{\bf t} + {\bf b}),$$
where $\ell$ is a linear function of ${\bf t}$ and ${\bf b}$ is a constant vector.
\end{thm}

\section{The three-lump example}

Consider a cuspidal singular curve with \emph{two} cusps with semigroup $\langle 2,5\rangle$ at the points $z = \pm 1$ in 
one affine coordinate on $\P^1$.  
By the constructions from \cite{ACL}, the $\langle 2,5\rangle$-cusps correspond to the triangular Young
diagrams from the partition $\lambda = (2,1)$.  Moreover, by Theorem~\ref{moredetail} the total degree of the polynomial analog
of the theta-function will be $3 + 3 = 6$ in this case.  We take a basis of $H^0(C,\omega_C)$ as follows
\begin{align}\label{dualizingdiffs}
\omega_1  &= \frac{4}{(u - 1)^2}\ du \nonumber \\
\omega_2  &= \frac{-12 u^2}{(u - 1)^4}\ du\\ 
\omega_3  &= \frac{4}{(u + 1)^2}\ du \nonumber\\
\omega_4  &= \frac{-12 u^2}{(u + 1)^4}\ du. \nonumber
\end{align}
(The constant factors $4,-12$ are not important here; they resulted from a different calculation to be described later.)

We can find an implicit equation of the 
$W_3$ subvariety of the generalized Jacobian of the cuspidal curve starting from the 
usual abelian integral parametrization.  We take the base point of the abelian integrals as the 
point $z = 0$ or $u = \infty$ on $\mathbb{P}^1$, the normalization of the cuspidal curve:
\begin{equation} \label{thetaparam}
Z_j = \int_\infty^{t_1} \omega_j + \int_\infty^{t_2} \omega_j + \int_\infty^{t_3} \omega_j, \quad j = 1,\ldots, 4.
\end{equation}
Eliminating the $t_i$  via a Gr\"obner basis calculation yields an implicit equation involving a polynomial 
analog of the theta-function:
\begin{align}\label{polytheta}
&Z_1^3  Z_3 ^3 + 24 Z_1^3 Z_3^2 - 24 Z_1^2 Z_3^3 + 192 Z_1^3 Z_3  + 16 Z_1^3 Z_4  - 540 Z_1^2 Z_3 ^2 + 192 Z_1  Z_3^3  \\
&+ 16 Z_2 Z_3^3 + 336 Z_1^3 - 4032 Z_1^2 Z_3 - 384 Z_1^2 Z_4 + 4032 Z_1 Z_3^2 + 384 Z_2 Z_3^2 - 336 Z_3^3 - 5904 Z_1^2\nonumber \\
& + 27936 Z_1 Z_3 + 3072 Z_1 Z_4 + 3072 Z_2 Z_3 + 256 Z_2 Z_4 - 5904 Z_3^2 + 32256 Z_1 + 5376 Z_2 \nonumber\\
&- 32256 Z_3 - 5376 Z_4 = 0. \nonumber
\end{align}
As expected by Theorem~\ref{moredetail}, the highest-degree term is the $Z_1^3 Z_3^3$ of total degree $6$.  

From Theorem~\ref{thetatotau} (Theorem 4.11 of \cite{ACL}), to produce a KP tau-function from this theta-function (up to 
an exponential factor that does not contribute anything when we apply \eqref{KPu} to produce the actual
KP1 solution $u(x,y,t)$), we need to 
compute a frame for the point of the Sato Grassmannian corresponding to the cuspidal curve.  In fact, 
we have already done the relevant computations needed here, since the crucial part of this comes
from the series expansions of the dualizing differentials from \eqref{dualizingdiffs} above.  We have
\begin{align}\label{framecoeffs}
\omega_1 &= (4 + 8u + 12 u^2 + 16 u^3 + \cdots\ )\ du \nonumber \\
\omega_2 &= (-12 u^2 - 48 u^3  + \cdots\ )\ du \\
\omega_3 &= (4 - 8u + 12u^2 - 16 u^3 + \cdots\ )\ du \nonumber\\
\omega_4 &= (-12 u^2 + 48 u^3 + \cdots\ )\ du\nonumber
\end{align}
Hence, taking 
\begin{align}\label{tausub}
Z_1 &= 4x + 8i y + 12 t + \phi_1\nonumber\\
Z_2 &= -12 t + \phi_2\\
Z_3 &= 4x - 8iy + 12 t + \phi_3\nonumber\\
Z_4 &=-12 t + \phi_4\nonumber
\end{align}
and substituting into \eqref{polytheta}, we obtain what is essentially 
a $\tau$-function for KP1 solution (note that the complex rescaling of $y$
to $i\cdot y$ has been incorporated 
here).   The $\phi_j$ are arbitrary constant parameters (analogous to ``phase factors'' for soliton solutions).  
They must be chosen appropriately to obtain a 
regular real solution.  In fact, for ``most'' values of the $\phi_j$ the resulting KP1 solutions produced by 
this recipe will still be non-regular, and 
they will also take non-real values at some real $(x,y,t)$.  We will see how to overcome these difficulties in the next section.

\section{Finding a real regular solution}

We begin with a simple observation related to the form of \eqref{polytheta} and \eqref{tausub}.  As long as 
$x, t$ are taken to be real, the only possible terms contributing non-real values in the substituted 
theta-function are those containing odd powers of $y$--that is, $y, y^3$, and $y^5$--together with arbitrary powers of $x,t$ 
giving a total degree at most $6$.  The coefficient of $y^5$ yields the terms
$$(-24 \phi_1 + 24  \phi_3 + 384) i$$
Hence, if $\phi_1 - \phi_3 = 16$, the coefficient of $y^5$ will be zero.  When $\phi_1 = 16 + \phi_3$ is
substituted into the coefficient of $y^3$, a rather surprising amount of cancellation happens and the 
only remaining terms are 
$$(44 - 2\phi_4 + 2 \phi_2) i.$$
If also $\phi_2 - \phi_4 = -22$, then the coefficient of $y^3$ is zero.  Moreover, these choices also make the coefficient of $y$
equal to zero, so the KP1 solution from \eqref{KPu} takes only real values for real $(x,y,t)$.  
The ``phase factors'' $\phi_3, \phi_4$ are still arbitrary, so to produce an explicit solution, we take $\phi_3 = \phi_4 = 0$.
The following is essentially a $\tau$-function for the solution we are considering:
\begin{align}  \label{lumptau}
& 198 x + (3033/2) x^2 t + (6615/2) x^2 t^2 + (8883/2) x t^2 + 1494 x t + 2592 y^2 x t^2 + x^6 + 6561 x t^3  \\
& + 306 x^2 y^2 + 576 y^4 t + 2970 y^2 t^2 + 3240 x^2 t^3 + 180 x^4 t + 96 x^3 y^2 + 1080 x^3 t^2  \nonumber\\
& + 741 x^3 t + 2592 y^2 t^3 + 1710 t y^2 + 4860 t^4 x + 522 x y^2 + 192 x y^4 + 1458 x t^5 + 432 y^4 t^2 \nonumber \\
& + 972 y^2 t^4 + 135 x^4 t^2 + 540 x^3 t^3 + 18 x^5 t + 1215 x^2 t^4 + 48 x^2 y^4 + 12 x^4 y^2 + 864 x^2 y^2 t \nonumber\\
 &+ 144 x^3 t y^2 + 648 x^2 t^2 y^2 + 1908 x y^2 t + 1296 x y^2 t^3 + 288 x y^4 t + 729 t^6 + 64 y^6 + 405 y^2 \nonumber\\
 & + 2916 t^5 + 228 y^4 + 12 x^5 + 261 x^2 + 2142 t^2 + 513/8 + (8667/2) t^3 + (345/2) x^3 \nonumber\\
 & + (1125/2) t + (249/4) x^4 + (19521/4) t^4. \nonumber
 \end{align}
(We say ``essentially'' because as always the actual $\tau$-function also includes an exponential factor that is linear in $x$, hence
does not contribute when \eqref{KPu} is applied.)

\begin{thm}  
The KP1 solution from the polynomial in \eqref{lumptau} is real and regular for all real $(x,y,t)$.
\end{thm}

\begin{proof}  That the $u(x,y,t)$ produced by \eqref{KPu} takes only real values for real $(x,y,t)$ is a consequence
of the determination of the $\phi_j$ described above and is also clearly visible from the form of \eqref{lumptau}.  
To show that this is a regular solution (a ``multi-lump'') we will show that the polynomial in \eqref{lumptau}
is a sum of squares of real polynomials with a nonzero constant term.  This will complete the proof.  To begin, 
we note that the polynomial analog of the theta-function from \eqref{polytheta} above can actually be 
rewritten in the form
\begin{align}\label{rearranged}
&(Z_1^3 - 24 Z_1^2 + 192 Z_1 + 16 Z_2 - 336)(Z_3^3 + 24 Z_3^2 + 192 Z_3 + 16 Z_4 + 336) \\
&+ 36 (Z_1 Z_3 + 10 Z_1 - 6 Z_3 - 56) (Z_1 Z_3 + 6 Z_1 - 10 Z_3 - 56). \nonumber
\end{align}
Part of this reflects the general patterns determined in the proof of Theorem 5.3 of \cite{ACL}.  The two factors of degree 3 on the first line are
actually the implicit equations of certain translates of the theta-divisors for the two partial normalizations of the bicuspidal curve,
where one of the cusps is smoothed and the other one remains untouched.  The other term is of lower total degree 4 and it happens
to factor in this way in this case.  When the values from \eqref{tausub} with $\phi_j$ as above are substituted into this polynomial, the two factors
of total degree 3 in the term on the first line become complex conjugates, and the product has the form 
$(A + iB)(A - iB) = A^2 + B^2$, where $A,B$ are polynomials with real coefficients.  The factors on the second line yield 
$$36((4 x + 12 t + 10)^2 + 64 y^2 + 4) ((4 x + 12 t + 6)^2 + 64 y^2 + 4)$$
which is also a sum of squares of real polynomials.  It is easy to see that the value of the whole polynomial when $x = y = t = 0$ is a strictly positive constant.
Hence the corresponding KP1 solution is regular for all real $(x,y,t)$.  It decays to zero as $(x^2 + y^2)^{-2}$ as $x^2 + y^2 \to \infty$.  
\end{proof}

We include some numerically-generated plots of the solution given by \eqref{lumptau} to demonstrate how it evolves over time.  

\begin{figure}
\centering
\includegraphics[width=3in,height=3in]{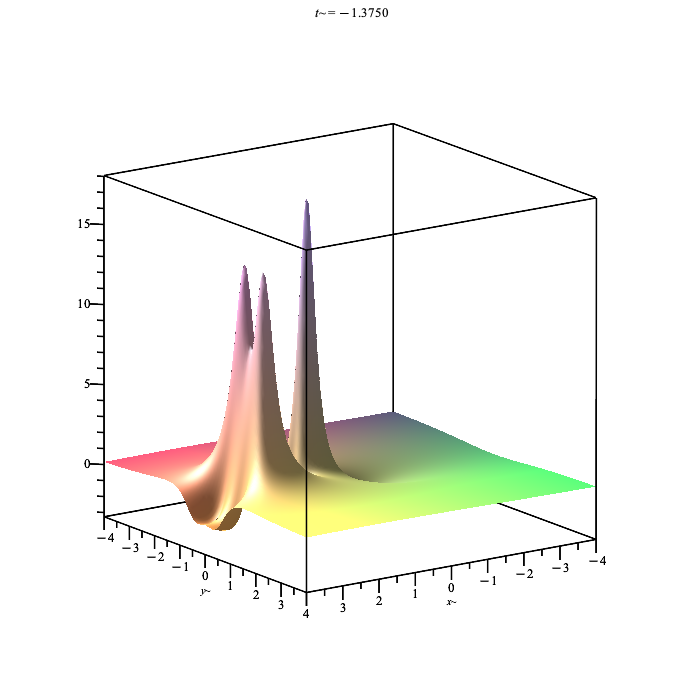}
\caption{The solution \eqref{lumptau} at $t = -1.375$.}
\end{figure}

\begin{figure}
\centering
\includegraphics[width=3in,height=3in]{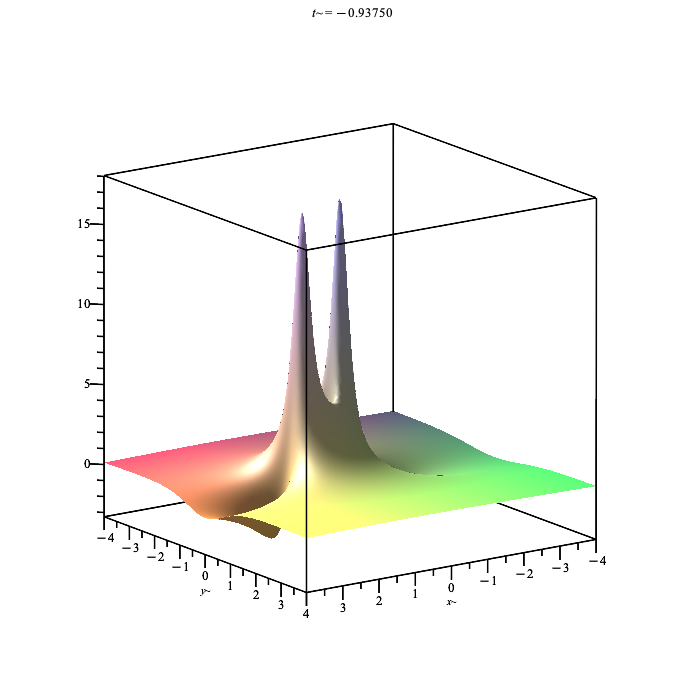}
\caption{The solution \eqref{lumptau} at $t = -0.9375$.}
\end{figure}

\begin{figure}
\centering
\includegraphics[width=3in,height=3in]{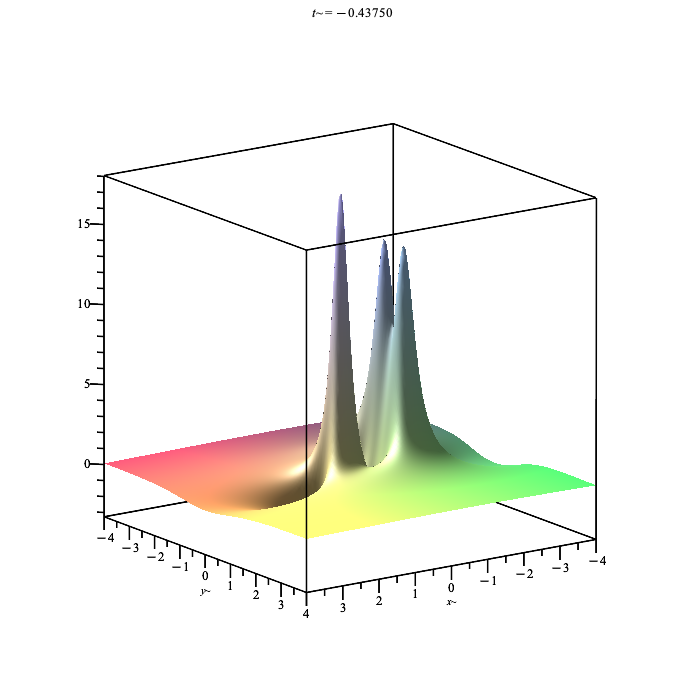}
\caption{The solution \eqref{lumptau} at $t = -0.4375$.}
\end{figure}

From Figures 1 and 3, we see that there are three local maxima (the ``lumps'') contained in the graph of $u(x,y,t)$ for $t$ outside a central region.  In Figure 2, however, 
we see that the lumps are undergoing an interesting interaction where two lumps have apparently coalesced and exchanged form with the taller single lump.
This means that our solution is \emph{not} the same as the traveling wave 3-lump solutions found in \cite{Zetal}.  The reason is that the authors
of \cite{Zetal} actually started from an analog of the Boussinesq equation in which dependence of $u$ on $t$ is omitted and a solution $v(x,y)$ of that equation is 
then used to generate a KP solution by setting $u(x,y,t) = v(d(x + Vt), y)$ where $d,V$ are constants and $V$ represents a wave speed.  Solutions like ours have been 
produced by other authors by different methods, though.  We refer the interested reader to the literature review and the bibliography of \cite{Zetal} for pointers to 
the relevant articles.

\section{Nodal curves and KP solitons}

To show how our method relates with other previous work producing multi-lump solutions, we will now describe 
an alternative way to derive the KP1 solution corresponding to \eqref{lumptau}.
We begin by setting up some suitable notation for understanding irreducible rational nodal curves and the corresponding soliton KP solutions.   
We follow notation from \cite{AFMS, FM}.
We start from $\mathbb{P}^1$ and identify $g$ pairs of points to produce an irreducible rational nodal
curve of arithmetic genus $g$.  In one of the standard coordinate charts of $\mathbb{P}^1$, say the points identified to give
the $i$th node are $z = b_i, c_i$.  Then a basis for the vector space of dualizing, or Rosenlicht, differentials on the nodal curve
is given by 
$$\omega_i = -\left(\frac{1}{z - b_i} - \frac{1}{z - c_i}\right)\ dz,$$
since the sum of the residues at $z=b_i$ and $z = c_i$ vanishes.  

For our purposes, it will be most convenient to change coordinates on $\mathbb{P}^1$, taking $u = \frac{1}{z}$ as the local 
coordinate at the point at infinity.  In terms of $u$, 
\begin{eqnarray} \label{nodaldiffs}
\omega_i &= {\displaystyle \left(\frac{1}{1/u - b_i} - \frac{1}{ 1/u - c_i}\right)\cdot \frac{1}{u^2}\ du}\\
&= {\displaystyle \frac{b_i - c_i}{(1 - b_i u)(1 - c_i u)}\ du}.
\end{eqnarray}
In terms of the coordinate $u$, these differentials can be expanded via geometric series in the form 
\begin{equation} \label{udiff}
\omega_i = \left((b_i - c_i) + (b_i^2 - c_i^2) u + (b_i^3 - c_i^3) u^2 + (b_i^4 - c_i^4) u^3 + \cdots \right)\ du
\end{equation}

It is well-known (see, for instance, \cite{M, AFMS, FM, K})  that the Riemann theta-functions on the Jacobians of a family of smooth curves degenerating to one of 
these rational nodal curves have a limit of the form 
$$\theta(z_1,\ldots,z_g) = \sum_{m \in \{0,1\}^g} \exp 2\pi i\left(\sum_{1 \le i < j\le g} m_i m_j \Omega_{ij} + \sum_{i=1}^g m_i z_i\right)$$
for some constants $\Omega_{ij}$.  These are the limits of the off-diagonal terms in the period matrices of the smooth curves, while the 
diagonal terms do not enter in the limit.  
It is also possible to determine a shift vector $h = (h_1,\ldots,h_g)$ such that the analog of the theta-divisor, that is, the $W_{g-1}$ subvariety
of the generalized Jacobian of the rational nodal curve, is given by the equation 
$$\theta(z_1 - h_1, \ldots, z_g - h_g) = 0$$
(the vector $h$ is analogous to the vector of ``Riemann constants'' which plays the same role for the theta-function from a smooth
genus $g$ curve).  

Soliton solutions of KP1 are then derived from these theta-functions first by 
substituting
$$z_j = (b_j- c_j)x + (b_j^2 - c_j^2) i\cdot y + (b_j^3 - c_j^3) t + \phi_j,$$
to yield KP1 $\tau$-functions and then applying \eqref{KPu}.
The $\phi_j$ are arbitrary constant ``phase factors'' whose values are then determined to produce real regular solutions.

In the article \cite{Zetal}, Zhang, Yang, Li, Guo, and Stepanyants give examples of degenerations of the $b_i$ and $c_i$ depending on 
a parameter $\varepsilon \to 0$ such that the lowest nonzero terms in the series expansion of the $\tau$-function in powers of $\varepsilon$ 
give rational functions of $x, y, t$ that yield ``lump'' solutions. 
 However, they do not make the connection with the cuspidal rational curves that are the limits of the nodal rational curves.
Hence the connection between their work and ours is that we will use 
similar degenerations, but we will show explicitly how the limit curve gives a cuspidal rational curve and how the limit
theta-function corresponds to the polynomial analog of the theta-function on the generalized Jacobian of the cuspidal curve.  

But in fact, to do this, it will be most convenient to look at the corresponding family of points in the Sato Grassmannian rather than looking
explicitly at the limit of the theta-function.  (That can also be done, of course, but the computations are more awkward.)
We will address these points in the next sections.

\section{Degenerating to the bi-cuspidal curve, and the polynomial analog of the theta-function}

From now on, we will specialize to the case $g = 4$ and a particular degeneration from a rational nodal curve to 
a cuspidal curve with two singular points.  The particular choice we will analyze will be the family of curves
constructed from $\mathbb{P}^1$ with the following four pairs of points identified
\begin{align} \label{degenpts}
b_1 = 1 + 2\varepsilon &\text{ and } c_1 = 1 - 2\varepsilon, \nonumber\\
b_2 = 1 + \varepsilon &\text{ and } c_2 = 1 - \varepsilon,\\
b_3 = -1 + 2\varepsilon &\text{ and } c_3 = -1 - 2\varepsilon,\nonumber\\
b_4 = -1 + \varepsilon &\text{ and } c_4 = -1 - \varepsilon.\nonumber
\end{align}
(These are the $z$-coordinates of the points, but we will usually pass to the other coordinate $u = \frac{1}{z}$ to work with
the dualizing differentials and the abelian integrals.)
We assume $\varepsilon$ is real and small enough in absolute value that all eight of these points of $\mathbb{P}^1$ 
are distinct.  For $\varepsilon \ne 0$, each pair of points yields a node; in the limit as $\varepsilon\to 0$, 
two nodes coalesce to a cuspidal singular point at $u = 1$ and the other two nodes coalesce to a cuspidal
singular point at $u = -1$.  This is a flat family and the total $\delta$-invariant of the singular points (that is, the 
arithmetic genus of the whole curve) is constant, equal to $g = 4$.  Each of the limit singular points is an $A_4$
double point with semigroup $\langle 2,5\rangle$.  For instance, a local planar model of the degeneration of one pair of nodes
to a $\langle 2,5\rangle$-cusp at $(x,y) = (0,0)$ is given by the family of parametrizations
\begin{align} \label{cuspparam}
x &= t^2\\
y &= t^5 - 5\varepsilon^2 t^3 + 4\varepsilon^4 t. \nonumber
\end{align}

We change basis in the vector space of dualizing differentials on the nodal curves of the family as follows to
obtain differentials with ``good'' limits on the cuspidal curve as $\varepsilon \to 0$.  We consider these linear
combinations of the differentials from \eqref{nodaldiffs} above:
\begin{align}\label{goodlimits}
\eta_1 &= \frac{1}{\varepsilon} \omega_1 \nonumber \\
\eta_2 &= \frac{1}{\varepsilon^3}(2 \omega_2 - \omega_1),\\
\eta_3 &= \frac{1}{\varepsilon} \omega_3, \text{ and } \nonumber \\
\eta_4 &= \frac{1}{\varepsilon^3}(2 \omega_4 - \omega_3). \nonumber
\end{align}
It is easy to check that the limits of the differentials from \eqref{goodlimits} as $\varepsilon\to 0$ exist and give
\begin{align}\label{thelimits}
\psi_1 = \lim_{\varepsilon\to 0} \eta_1 &= \frac{4}{(u - 1)^2}\ du \nonumber \\
\psi_2 = \lim_{\varepsilon\to 0} \eta_2 &= \frac{-12 u^2}{(u - 1)^4}\ du\\ 
\psi_3 = \lim_{\varepsilon\to 0} \eta_3 &= \frac{4}{(u + 1)^2}\ du \nonumber\\
\psi_4 = \lim_{\varepsilon\to 0} \eta_4 &= \frac{-12 u^2}{(u + 1)^4}\ du. \nonumber
\end{align}
These are the same as the differentials from \eqref{dualizingdiffs}.
In terms of the original affine coordinate $z = \frac{1}{u}$, these can be 
written as 
$$\psi_1 = \frac{-4}{(z - 1)^2}\ dz, \psi_2 = \frac{12}{(z - 1)^4}\ dz, \psi_3 = \frac{-4}{(z + 1)^2}\ dz, \psi_4 = \frac{12}{(z + 1)^4}\ dz,$$
which is exactly the expected form for the dualizing differentials on a curve with 
two $\langle 2,5\rangle$-cusps.  (Note that as observed in \cite{ACL}, the exponents after integrating
with respect to $z$ would correspond to the ``gaps'' ${1,3}$ of this semigroup.)

Then the rest of the derivation follows as before, in particular the use of Theorem~\ref{thetatotau} to derive the KP tau-function.

\section{Comments and Generalizations}

By analogy with other physical situations, multi-lump solutions of KP1 such as the one from \eqref{lumptau} have been called ``bound states''
in \cite{Zetal} and elsewhere.  
There are also other configurations of lumps that have been constructed by other methods and we do not understand whether or how
the ideas here might be applied to all those other sorts of examples yet.

On a more hopeful note, 
much of the construction we have presented, starting from the choice of the pairs of points from \eqref{degenpts} for the family of nodal
curves  generalizes immediately to give families degenerating to cuspidal curves with two $A_{2N}$ double points for 
all $N \ge 1$.  The forms of the dualizing differentials on the cuspidal limits will be parallel, and the results of \cite{ACL}
yield a similar factorization of the leading terms of the polynomial analog of the theta-function on the cuspidal curve.  Theorem 
4.11 of \cite{ACL} applies in general as here to give KP1 solutions.

\begin{conj}
We conjecture that the methods used to generate this example will generalize to give similar real regular 
$M$-lump solutions with $M = \frac{N(N+1)}{2}$ for all $N \ge 1$  starting
from rational curves with two cusps with semigroup $\langle 2,2N+1\rangle$ ($A_{2N}$ double points). 
We expect these solutions will have triangular configurations of lumps with $N$ rows containing $1,2,\ldots, N$ lumps
respectively.  
\end{conj}

In additional support of this conjecture, we include a plot of a similar solution constructed from a bicuspidal curve with two $\langle 2,7\rangle$
cusps ($A_6$ double points) following the same plan as that used in  the calculations reported above.   This can be obtained from 
a family of rational nodal curves as in \eqref{degenpts}, but with
\begin{align} \label{degenpts6}
b_1 = 1 + 3\varepsilon &\text{ and } c_1 = 1 - 3\varepsilon, \nonumber\\
b_2 = 1 + 2\varepsilon &\text{ and } c_2 = 1 - 2\varepsilon,\nonumber\\
b_3 = 1 + \varepsilon &\text{ and } c_3 = 1 - \varepsilon,\\
b_4 = -1 + 3\varepsilon &\text{ and } c_4 = -1 - 3\varepsilon,\nonumber\\
b_5 = -1 + 2\varepsilon &\text{ and } c_5 = -1 - 2\varepsilon.\nonumber\\
b_6 = -1 + \varepsilon &\text{ and } c_6 = -1 - \varepsilon,\nonumber
\end{align}
The polynomial analog of the theta-function
in this case has degree 12.  The polynomials involved are too complicated  to be readily understandable, so they are omitted.  The pattern
established in Theorem 5.3 of \cite{ACL} is clear, though.  The highest degree term in the polynomial analog of the theta-function
is $Z_1^6 Z_4^6$ since the corresponding partition for each cusp is the triangular $\lambda = (3,2,1)$.  

\begin{figure}
\centering
\includegraphics[width=3in,height=3in]{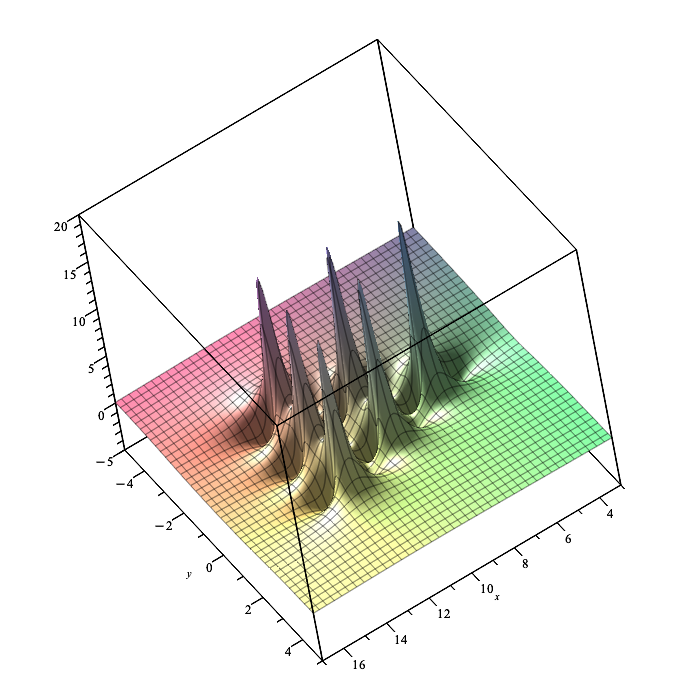}
\caption{The case $N = 3$---another KP1 solution with 6 lumps.}
\end{figure}

The plot in Figure 4 shows the surface of $u(x,y,t)$ from above for the one value $t = -3$ so the lump arrangement is visible.      The lumps seem to coalesce for $t$ around $0$, then emerge in a reflected version
of this same pattern as $t$ increases.

Finally, in Figure 5, we display a time snapshot of another five-lump solution constructed from a rational curve with two cusps 
having semigroup $\langle 3,4\rangle$.  
\begin{figure}
\begin{center}
\includegraphics[width=3in,height=3in]{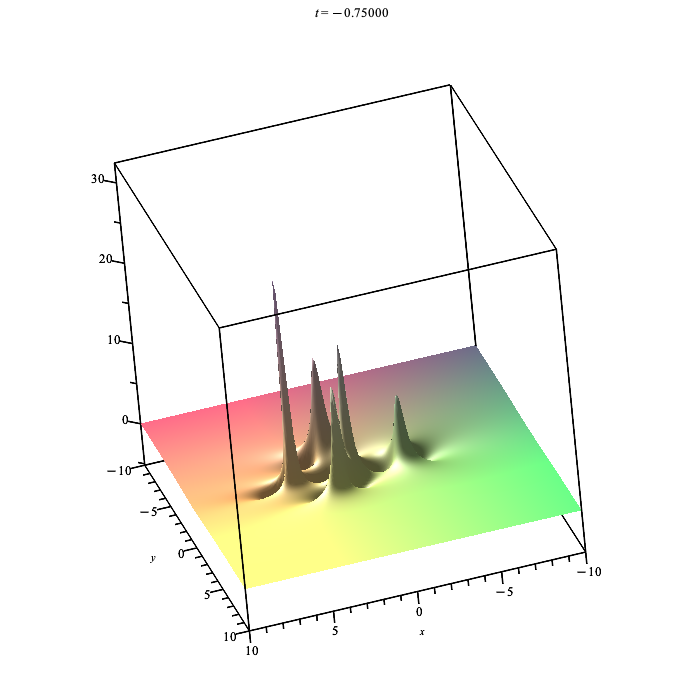}
\end{center}
\caption{A KP1 solution with 5 lumps.}
\end{figure}
The corresponding partitions are both $\lambda = (3,1,1)$ and the polynomial analog of the theta-function has degree $10 = 2 \cdot (3 + 1 + 1)$.  

The obstacle that we have not overcome as of yet to generalize these results 
is proving that there will always be choices of the ``phase factors'' $\phi_j$ as above that produce 
real regular KP1 solutions.  This might be obtained with more detailed information about the form of the terms of lower total degree in the polynomial analog
of the theta-function on the bicuspidal curve.  This seems somewhat similar to, but more complicated than, the form seen in  \eqref{rearranged} when
the degree is larger.  We hope to return to this in the future.

\bibliographystyle{plain}

\end{document}